%% file: Parini_Stylianou_2017_arXiv.tex
\documentclass[12pt,a4paper]{article}
\usepackage{hyperref}
\usepackage{csquotes}
\input{preamble.tex}
\begin{document}
\selectlanguage{english}

\title{A free boundary approach to the Rosensweig instability of ferrofluids}
\author{E. Parini\thanks{Aix Marseille Univ, CNRS, Centrale Marseille, I2M, 39 Rue Frederic Joliot Curie, 13453 Marseille, France}
\ and\ A. Stylianou\thanks{Institut f\"{u}r Mathematik, Universit\"{a}t Kassel, 34132 Kassel, Germany}
}
\maketitle

\begin{abstract}
We establish the existence of saddle points for a free boundary problem describing the two-dimensional free surface of a ferrofluid which undergoes normal field instability (also known as Rosensweig instability). The starting point consists in the ferro-hydrostatic equations for the magnetic potentials in the ferrofluid and air, and the function describing their interface. The former constitute the strong form for the Euler-Lagrange equations of a convex-concave functional. We extend this functional in order to include interfaces that are not necessarily graphs of functions. Saddle points are then found by iterating the direct method of the calculus of variations and by applying classical results of convex analysis. For the existence part we assume a general (arbitrary) non linear magnetization law. We also treat the case of a linear law: we show, via convex duality arguments, that the saddle point is a constrained minimizer of the relevant energy functional of the physical problem.
\end{abstract}


\section{Introduction}
\subsection{The ferro-hydrostatic equations}
Let $\Omega\subset \mathbb R^2$ be open, connected and bounded with a Lipschitz continuous boundary, and $b,\tau>0$. Moreover, let $\mu\in C_b^1(\mathbb R)$, and for a function $\eta:\mathbb R^2\strongly\mathbb R$ define the sets
\begin{equation}\label{boundaries}
\begin{aligned}
 D_\eta^+ {}& \defeq \left\{(x,y,z)\in\mathbb R^3: (x,y)\in\Omega\ \text{ and }\ z\in\left(\eta(x,y),1\rule{0pt}{10pt}\right)\right\},\\
 D_\eta^- {}& \defeq \left\{(x,y,z)\in\mathbb R^3: (x,y)\in\Omega\ \text{ and }\ z\in\left(-1,\eta(x,y)\rule{0pt}{10pt}\right)\right\},\\
 \partial_{cyl} D {}& \defeq \left\{(x,y,z)\in\mathbb R^3: (x,y)\in\partial\Omega\ \text{ and }\ z\in\left(-1,1\right)\right\}\\
 \partial_{top} D {}& \defeq \left\{(x,y,z)\in\mathbb R^3: (x,y)\in\Omega\ \text{ and }\ z=1\right\},\\
 \partial_{bot} D {}& \defeq \left\{(x,y,z)\in\mathbb R^3: (x,y)\in\Omega\ \text{ and }\ z=-1\right\},
\end{aligned}
\end{equation}
\begin{figure}[t]
 \centering\includegraphics[scale=.7]{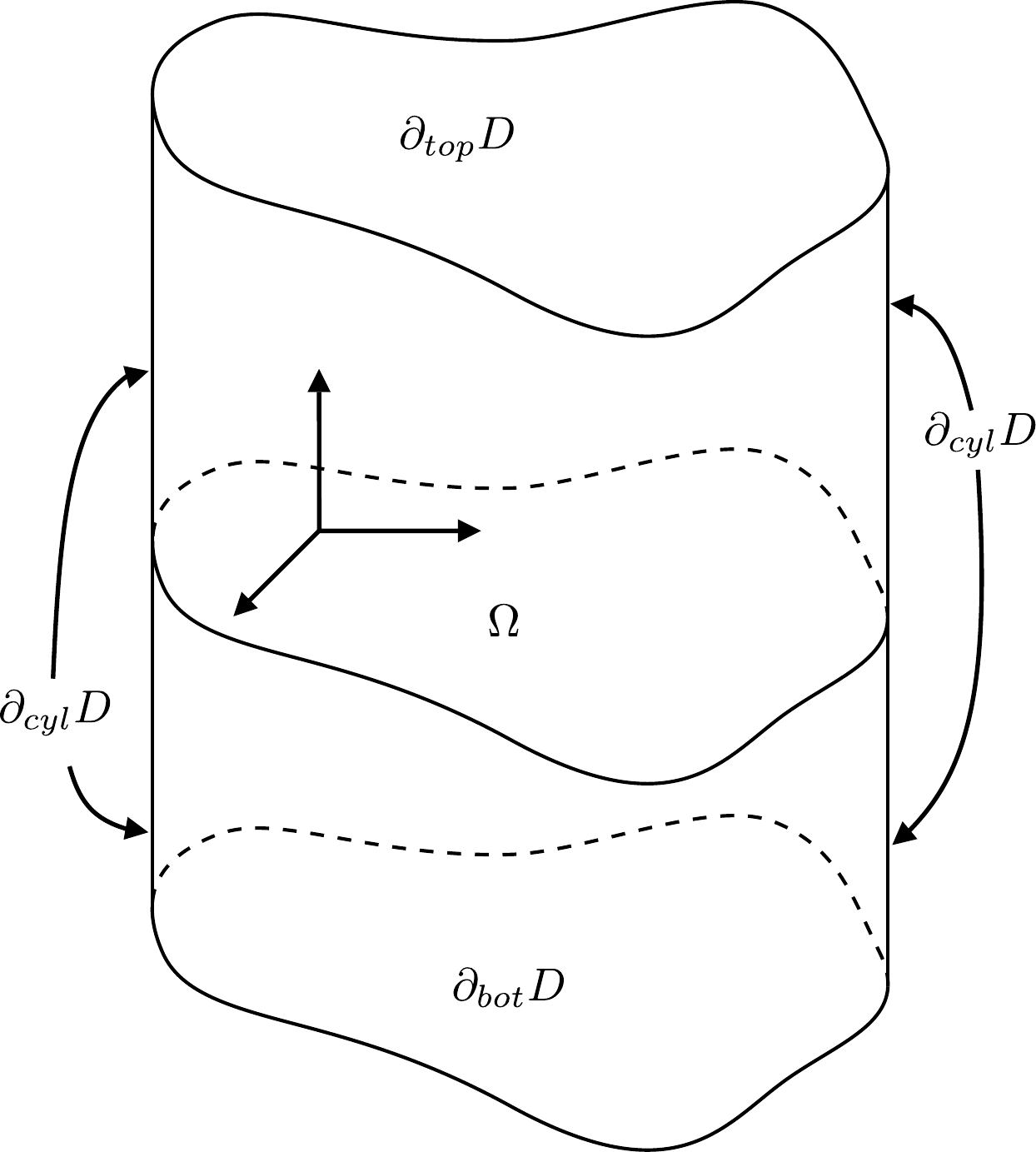}
 \caption{The cylindrical domain in which the ferrofluid is contained. The domain $\Omega$ lies at $z=0$ and corresponds to the undisturbed surface of the ferrofluid before the magnet is turned on.}
\end{figure}
and the function
\begin{equation}\label{primitive}
 M(s)\defeq\int_0^s t\cdot\mu(t)\;dt.
\end{equation}
Consider the following problem: Find sufficiently smooth functions $\phi,\psi:\mathbb R^3\strongly\mathbb R$ and $\eta:\mathbb R^2\strongly\mathbb R$ with $|\eta|<1$ that satisfy the boundary value problem
\begin{equation}\label{non_hom_ferroPDE}
 \begin{aligned}
  \displaystyle\Delta \psi {}& =  0 & \text{ in }{}& D_\eta^+,\\
  \displaystyle\diver\left(\mu\left(|\nabla\phi|\right)\nabla\phi\right) {}&= 0 & \text{ in } {}&D_\eta^-,\\
  \displaystyle\psi_z{}& = \mu(1) & \text{ on } {}&\partial_{top}D,\\
  \displaystyle\mu\left(|\nabla\phi|\right)\phi_z {}&= \mu(1) & \text{ on }{}& \partial_{bot} D,\\
  \displaystyle\psi {}& =  0 & \text{ on }{}& \partial_{cyl}D,\\
  \displaystyle\phi {}& =  0 & \text{ on }{}& \partial_{cyl}D,\\
 \end{aligned}
\end{equation}
together with the compatibility conditions
\begin{equation}\label{compatibility_conditions}
 \begin{aligned}
  \phi{}&=\psi{}& \text{ on } {}&\Omega\times \{z=\eta(x,y)\}\\
  \mu\left(|\nabla\phi|\right)\phi_{\mathbf n}{}&=\psi_{\mathbf n}{}& \text{ on } {}&\Omega\times \{z=\eta(x,y)\},
 \end{aligned}
\end{equation}
and the free surface equation
\begin{equation}\label{free_surface_equation}
\begin{aligned}
  M\left(|\nabla\phi|\right)-\frac{1}{2}|\nabla\psi|^2+\sqrt{1+|\nabla\eta|^2}\left(\rule{0pt}{10pt}\psi_z\psi_{\mathbf n}-\mu(|\nabla\phi|)\phi_z\phi_{\mathbf n}\right)\\
  +\tau\diver \frac{\nabla \eta}{\sqrt{1+|\nabla\eta|^2}}-b\eta-p_0{}&=0,
\end{aligned}
\end{equation}
on $\Omega\times \{z=\eta(x,y)\}$. Here,
\[
 \mathbf n\defeq\frac{1}{\sqrt{1+|\nabla\eta|^2}}(-\eta_x,-\eta_y,1),
\]
is the unit normal of the surface $\Omega\times\{z=\eta(x,y)\}$ in the direction of positive $z$, and the constant $p_0\defeq M(1)+\mu(1)\left(\frac{1}{2}\mu(1)-1\right)$.

Finally we would like to point out the following notational convention: the dimension of all operators appearing in the paper conform to the dimension of the domain of definition of their arguments, that is,
$$\nabla \phi=(\phi_x,\phi_y,\phi_z),\ \diver \eta=\eta_x+\eta_y,\ |\nabla\phi|=(\phi_x^2+\phi_y^2+\phi_z^2)^{1/2},\ |\nabla\eta|=(\eta_x^2+\eta_y^2)^{1/2}, \text{ etc.}$$

\subsection{Physical attributes and modelling of the normal field instability of a ferrofluid}
The system of partial differential equations described in the previous subsection arises as the mathematical model of an incompressible ferrofluid undergoing the so-called \textit{normal field instability} or \textit{Rosensweig instability} (refer for example to Rosensweig's monograph \cite{Rosensweig1985}): In an experiment, a vertical magnetic field is applied to a static ferrofluid layer, and various patterns (typically regular cellular hexagons) emerge on the fluid surface as the field strength is increased through a critical value.

Note that the strength of the applied field does not seem to appear in the system; it is rescaled to $1$. The function $\eta$ defines the (rescaled) interface between the ferrofluid and air (or another fluid conforming to a linear magnetization law), that is, they occupy the regions $D_\eta^-$ and $D_\eta^+$ respectively and are subjected to a parallel vertical magnetic field; for more details on extracting the ferrofluid system from Maxwell's equations and the appropriate rescaling of the initial physical laws see \cite{GrovesLloydEtAl2017} and references therein. The real functions $M$ and $\mu$ describe the magnetization law for the ferrofluid and the unknown functions $\phi,\psi$ are the magnetic potentials in the ferrofluid and air respectively. A typical case consists in ferrofluids following a nonlinear Langevin law (see \cite{GollwitzerLloydEtAl2015,RichterLange2009}), that is,
 \begin{equation}
  \label{mu}\mu(s)=\left\{ 
  \begin{aligned}
   & 1+\frac{M_s}{s}\left(\coth(\gamma s)-\frac{1}{\gamma s}\right),& \text{for } s\neq0,\\
   & 1 + \frac{M_s \gamma}{3},& \text{for } s=0,
  \end{aligned}
  \right.
 \end{equation}
where $M_s$ is  the saturation magnetization and $\gamma$ is the Langevin parameter (they are both positive constants). Lastly, the parameters $b$ and $\tau$ are respectively the rescaled gravity acceleration and the coefficient of surface tension for the ferrofluid.

Since the invention of the ferrofluids (\cite{Papell1965}) there has been a number of works studying surface instabilities using formal analysis (see \cite{GrovesLloydEtAl2017} and references therein). A first rigorous treatment of regular patterns assuming a linear magnetization law was given by Twombly and Thomas \cite{TwomblyThomas1983}. For small amplitude localized patterns and nonlinear laws, Groves \textit{et al} produced a rigorous theory in \cite{GrovesLloydEtAl2017}, using a technique known as \textit{Kirchg\"assner reduction}.

This work lies between the latter two papers in the following sense: we allow for general nonlinear laws and simultaneously pose no assumption on the smallness of solutions. This is achieved through the study of the problem as a free boundary problem, that is, the originally unknown function $\eta$ that models the free interface of the ferrofluid is replaced by the characteristic function of the set occupied by the ferrofluid. The new set of unknowns (the magnetic potentials of the two fluids and the characteristic function of the ferrofluid) is then found as a critical point of an appropriate functional.

\subsection{The variational structure of the ferrofluid system}
The equations \eqref{non_hom_ferroPDE}-\eqref{free_surface_equation} describing the ferrofluid system have a variational structure: they correspond to the strong Euler-Lagrange equations of the functional
\begin{equation}
 \label{functional}
 \begin{aligned}
  F(u,\eta)  = {}& \int_\Omega\left( \int_{-1}^{\eta(x,y)}M(|\nabla u|)\;dz+ \int_{\eta(x,y)}^1\frac{1}{2}|\nabla u|^2\;dz\right)dxdy\\[0.5ex]
  {}&+\mu(1)\int_\Omega \left(u|_{z=-1}-u|_{z=1}\right)dxdy\\[0.5ex]
  {}&-\int_\Omega\left(\frac{b}{2}\eta^2+p_0\eta\right)dxdy-\tau\int_\Omega \sqrt{1+|\nabla\eta|^2}\,dxdy.
 \end{aligned}
\end{equation}
This means that, assuming that $( u ,\eta)$ is a critical point of $F$ and allowing
\[
 \psi= u |_{D_\eta^+}\in W^{2,2}\left(D_\eta^+\right)\ \text{ and }\ \phi= u |_{D_\eta^-}\in W^{2,2}\left(D_\eta^-\right),
\]
and for some smoothness of the interface, for example $\eta\in C^1$ and $|\eta|<1$, one obtains a strong solution to the ferrofluid system of equations. 

Note that $F$ is not a ``pure'' energy functional: it is convex with respect to $u$ and ``almost'' concave with respect to $\eta$ (the variable integral is not affine but is still bounded). This implies that it does not possess minimizers: taking a highly oscillating interface will produce ``energies'' that tend to $-\infty$ as the oscillations increase.

\begin{remark}
 The physical parameters $b$ and $\tau$ are of the order of $H^{-2}$, where $H$ denotes the strength of the applied field (see \cite{GrovesLloydEtAl2017}). Thus, dividing \eqref{functional} by $-H^2$, letting $H \strongly 0$, and assuming incompressibility of the fluid, we are led to the problem of minimizing the functional
 $$ \eta\mapsto \frac{\tilde b}{2}\int_\Omega\eta^2\;dxdy+\tilde \tau\int_\Omega \sqrt{1+|\nabla\eta|^2}\;dxdy.$$
 Due to the absence of side conditions, the minimizer is the function $\eta=0$. This justifies physical intuition that the surface of the ferrofluid remains flat in the absence of external field.
\end{remark}

The next observation concerns the mathematical properties of the nonlinear magnetization law. The following lemma can be proven with elementary analytical arguments.
\begin{lemma}\label{mulemma}
 Let $\mu,M:\mathbb R\strongly \mathbb R$ be defined by \eqref{primitive} and \eqref{mu}.
\begin{enumerate}
 \item $\mu$ is an even function.
 \item $\mu\in C(\mathbb R)$ and $1<\mu(s)\leq 1+\frac{M_s \gamma}{3}$ with equality if and only if $s=0$.
 \item For all $ s\geq0$ holds that
 \[\frac{1}{2}s^2\leq M(s)\leq \frac{1}{2}\left(\frac{\gamma}{3}M_s+1\right)s^2,\]
 with equality if and only if $s=0$.
 \item $M$ is a convex function.
\end{enumerate}
\end{lemma}
\begin{proof}[Sketch of proof]
 $1.$ and $2.$ follow from the explicit expression of $\mu$. 3. follows from $2.$, and $4.$ from the fact that a primitive is explicitly computable for a Langevin law: $M(s)=s^2/2+M_s\,\big(\ln\big(\sinh(\gamma s)\big)-\ln s-\ln \gamma\big)/\gamma$.
\end{proof}
\begin{figure}[ht]
 \centering
 \begin{tabular}{cc}
  \includegraphics[width=0.45\textwidth]{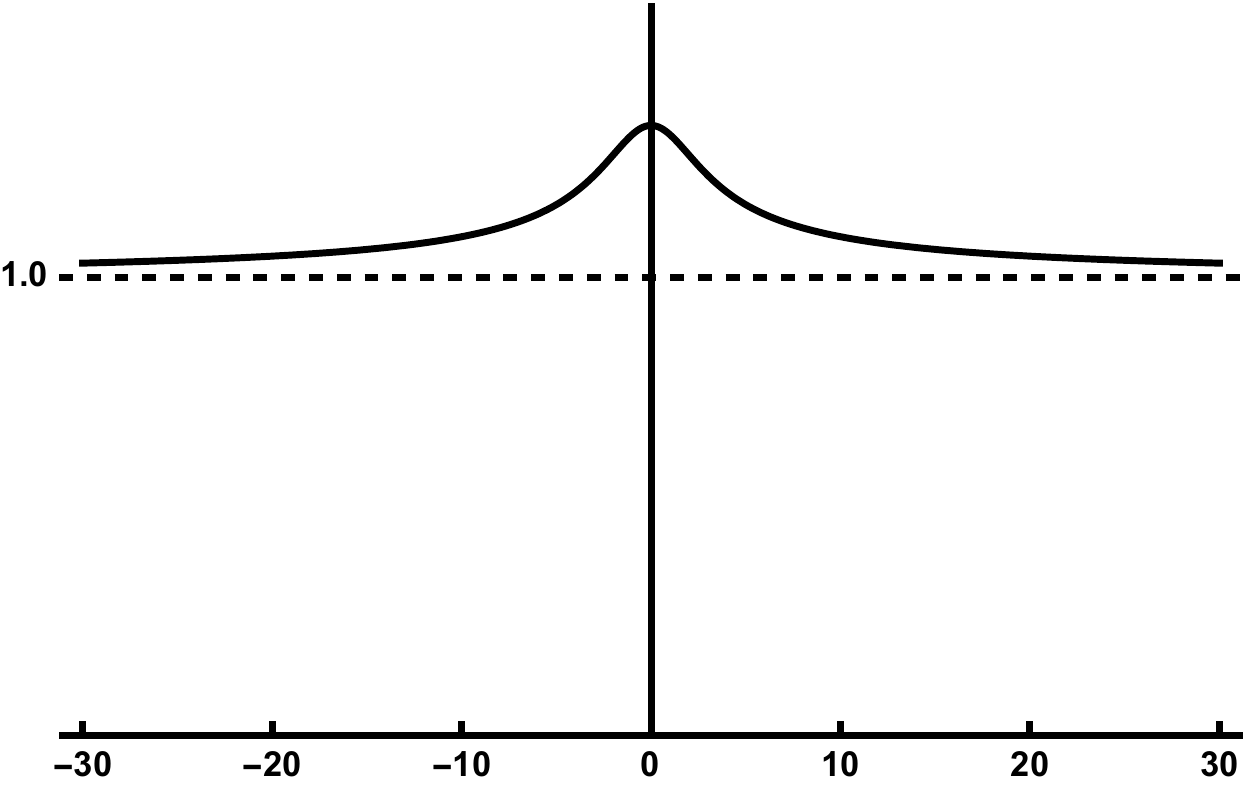} &\quad \includegraphics[width=0.45\textwidth]{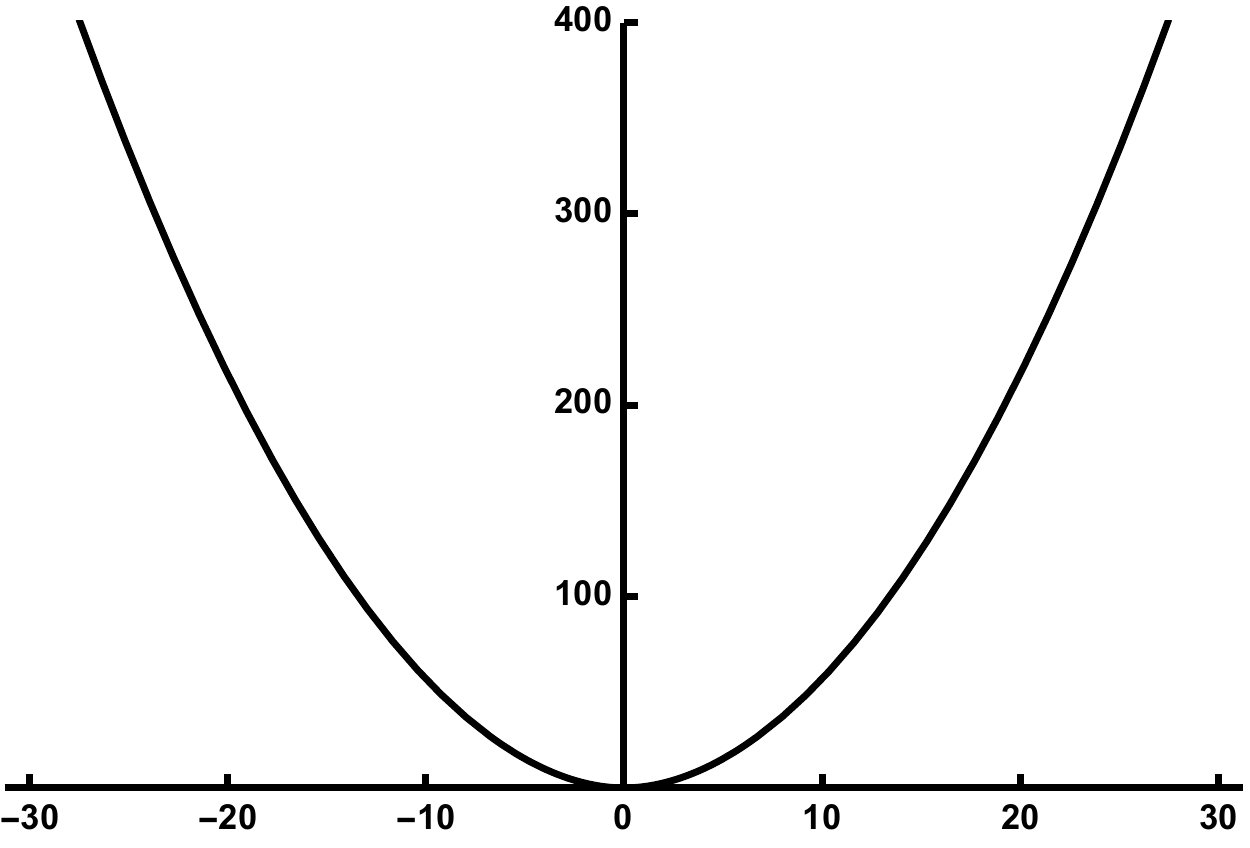}
 \end{tabular}
 \caption{A plot of the function $\mu$ defined as in \eqref{mu} with $\gamma=M_s=1$ \textit{(left)} and of its primitive \textit{(right)}.}
\end{figure}

This has two direct consequences: First, we have that $p_0>0$ since the expression $\mu(1)\left(\frac{1}{2}\mu(1)-1\right)$ has an infimum value of $-\frac{1}{2}$ for $\mu(1)\strongly 1$ (which is not attained since $\mu(1)>1$) and $M(1)>\frac{1}{2}$. Second, we can define the domain of \eqref{functional}; the functional is well defined for $ u \in H^1(D)$ and $\eta \in BV(\Omega)$.
\section{The mathematical setting}
\subsection{A more general framework}
From now on, and in order to shorten the formulas, we will use the notation $d\mathbf x$ for the Lebesgue measure in $\mathbb R^3$. We will drop the assumption that the interface can be described by the graph of a function $\eta$ and use only the minimal assumptions needed for the magnetization of the ferrofluid. Motivated by the properties of generic Langevin laws, we pose the following:
\begin{assumptions}\label{assume}
 \begin{enumerate}
  \item The physical constants $b,\tau,\mu,p_0\in\mathbb R$ are positive.
  \item The function $M:\mathbb R\strongly\mathbb R$ is convex and there exists $C_M>1$ such that
  \begin{equation}
   \label{assumeM}\frac{1}{2}s^2\leq M(s)\leq \frac{C_M}{2} s^2\ \text{ for all $s\in\mathbb R$}.
  \end{equation}
 \end{enumerate}
\end{assumptions}
Concerning the magnetic potential, define the space
 \begin{equation}
  \label{Wcyl}H^1_{cyl}(D)\defeq\big\{ u \in H^1(D): u |_{\partial_{cyl}D}=0\big\}
 \end{equation}
equipped with the norm $\lnorm u\rnorm_{cyl}\defeq \lnorm \nabla u\rnorm_{L^2}$ which, due to Poincar\'e's inequality, is equivalent to the standard Sobolev norm. In contrast to \cite{GrovesLloydEtAl2017}, we address the problem in a weaker form, namely as a free boundary problem. To that end, define the set of characteristic functions
\begin{equation}
 \label{X}  X (D)\defeq \bigg\{\chi\in BV(D): \chi\in\{0,1\}\text{ in } D,\ \int_D \chi\;d\mathbf x=|\Omega|\bigg\}
\end{equation}
and note that it is a weakly-$*$ closed subset of $BV(D)$ (every $L^1$ convergent sequence has an a.e. convergent subsequence and thus the limit function will be a.e. either $0$ or $1$. The integrals then converge to the correct value due to the dominated convergence theorem). The volume condition is due to the fact that the ferrofluid is assumed to be incompressible.

We consider the functional $J:H^1_{cyl}(D)\times  X (D)\strongly \mathbb R$ defined by
\begin{equation}
 \label{GeneralFunctional}
 \begin{aligned}
  J( u ,\chi)  \defeq {}& \int_D \left(\chi\,M(|\nabla u |)+\frac{1}{2}(1-\chi)\,|\nabla u |^2\right) d\mathbf x+\mu\int_\Omega \left( u |_{z=-1}- u |_{z=1}\right)dxdy\\[5pt]
  {}&-\int_D\left(b\,z\,\chi+p_0\,\chi\right)d\mathbf x-\tau\int_D |\nabla\chi|,
 \end{aligned}
\end{equation}
where $\int_D |\nabla\chi|$ denotes the total variation of $\chi$, and note that 
\begin{equation}
 \label{GeneralFunctional1}
 J( u ,\chi)=J_1( u ,\chi)-J_2(\chi)-\mu\int_D  u _z\;d\mathbf x,
\end{equation}
where
\begin{align}
 J_1( u ,\chi)  \defeq {}& \int_D \left(\chi\,M(|\nabla u |)+\frac{1}{2}(1-\chi)\,|\nabla u |^2\right) d\mathbf x,\\
 J_2(\chi)\defeq {}& \int_D\left(b\,z\,\chi+p_0\,\chi\right)d\mathbf x+\tau\int_D |\nabla\chi|.
\end{align}

The reason to consider the above functional lies in the following: as already mentioned, critical points of \eqref{functional} are weak solutions of the ferrofluid problem \eqref{non_hom_ferroPDE}-\eqref{free_surface_equation}. The functional $J$ defined above is an appropriate extension of \eqref{functional} where we have dropped the assumption that the interface can be described by the graph of $\eta$. The function $\eta$ is replaced by a function $\chi$ which, as a characteristic function, yields the set that is occupied by the ferrofluid. Precisely, for a function $\eta\in BV(\Omega)$ holds (see eg. \cite[Theorem 16.4, p.163]{Giusti1984})
\begin{equation*}
 J( u ,\chi_{\{z<\eta\}})=F( u ,\eta)-\frac{b}{2}|\Omega|.
\end{equation*}

All in all, critical points $( u ,\chi)$ with $\chi$ being the characteristic function of some open $D_F\subset D$ with appropriately smooth boundary will satisfy problem \eqref{non_hom_ferroPDE}-\eqref{free_surface_equation} locally, in the sense that the part of $\partial D_F$ that lies inside $D$ is locally the graph of $\eta$, and 
\begin{equation*}\label{subst}
\psi\defeq  u |_{D\setminus D_F}\ \text{ and } \ \phi\defeq  u |_{D_F}
\end{equation*}
satisfy the differential equations in a weak sense. These critical points are sought as saddle points of $J$ since the functional is now convex-concave: $J_2$ is convex, $J_1(\cdot,\chi)$ is convex and $J_1( u ,\cdot)$ is affine.

\subsection{The framework of abstract minimax theory of convex-concave functions}
Before we proceed with the exposition we give some definitions, following the analysis of \textit{saddle} or \textit{convex-concave} functions given in \cite{BarbuPrecupanu2012} and \cite{PapageorgiouKyritsi2009}.
\begin{definition}
 Let $\mathrm X $ and $\mathrm Y $ be two nonempty sets and $\Phi: \mathrm X \times \mathrm Y \strongly \mathbb R\cup \{+\infty\}$ a function. We say that $(\mathrm x_0,\mathrm y_0)$ is a \emph{saddle point} of $\Phi$, if we have
 \begin{equation}
  \Phi(\mathrm x_0,\mathrm y)\leq\Phi(\mathrm x_0,\mathrm y_0)\leq \Phi(\mathrm x,\mathrm y_0)
 \end{equation}
 for all $(\mathrm x,\mathrm y)\in \mathrm X \times \mathrm Y $.
\end{definition}
\begin{definition}
 Let $\mathrm X $ and $\mathrm Y $ be two nonempty sets and $\Phi: \mathrm X \times \mathrm Y \strongly \mathbb R\cup \{+\infty\}$ a function. We say that $\Phi$ has a \emph{saddle value} $c$ if
 \begin{equation}
  \adjustlimits\sup_{\mathrm y\in \mathrm Y }\inf_{\mathrm x\in \mathrm X }\Phi(\mathrm x,\mathrm y)=\adjustlimits\inf_{\mathrm x\in \mathrm X }\sup_{\mathrm y\in \mathrm Y }\Phi(\mathrm x,\mathrm y)\eqdef c.
 \end{equation}
\end{definition}
\begin{definition}
 Let $\mathrm X $ and $\mathrm Y $ be two nonempty sets and $\Phi: \mathrm X \times \mathrm Y \strongly \mathbb R\cup \{+\infty\}$ a function. We say that $\Phi$ \emph{satisfies a minimax equality} at $(\mathrm x_0,\mathrm y_0)\in \mathrm X \times \mathrm Y $ if:
 \begin{enumerate}
  \item The function $\Phi$ has a saddle value.
  \item There exists $\mathrm x_0\in \mathrm X $ such that $\displaystyle\sup_{\mathrm y\in \mathrm Y }\Phi(\mathrm x_0,\mathrm y)=\adjustlimits\inf_{\mathrm x\in \mathrm X }\sup_{\mathrm y\in \mathrm Y }\Phi(\mathrm x,\mathrm y)$.
  \item There exists $\mathrm y_0\in \mathrm Y $ such that $\displaystyle\inf_{\mathrm x\in \mathrm X }\Phi(\mathrm x,\mathrm y_0)=\adjustlimits\sup_{\mathrm y\in \mathrm Y }\inf_{\mathrm x\in \mathrm X }\Phi(\mathrm x,\mathrm y)$.
 \end{enumerate}
\end{definition}
One can then directly prove the following fundamental result.
\begin{proposition}[{\cite[Proposition 2.105]{BarbuPrecupanu2012},\cite[Proposition 2.3.5]{PapageorgiouKyritsi2009}}]\label{minimaxEq}
 Let $\mathrm X $ and $\mathrm Y $ be two nonempty sets and $\Phi: \mathrm X \times \mathrm Y \strongly \mathbb R\cup \{+\infty\}$. The function $\Phi$ satisfies a minimax equality at $(\mathrm x_0,\mathrm y_0)$ if and only if $(\mathrm x_0,\mathrm y_0)$ is a saddle point of $\Phi$ in $\mathrm X \times \mathrm Y $.
\end{proposition}

We outline the strategy for proving existence of saddle points: The special type of coupling between $u$ and $\chi$ that occurs only in $J_1$ allows us to iteratively apply the direct method of the calculus of variations to solve the min-max and the max-min problem. The next step is to show that the functional has a saddle value, that is, the max-min and the min-max problems are solved at the same value. To that end, we will use a tool from \cite{PapageorgiouKyritsi2009}, a ``coincidence theorem'', which is a corollary of a classical result of Knaster, Kuratowski and Mazurkiewicz \cite{KnasterKuratowskiEtAl1929}; we provide it here without a proof.
\begin{proposition}[{\cite[Proposition 2.3.10]{PapageorgiouKyritsi2009}}]\label{KKM}
 Assume that $\mathbf X$ and $\mathbf Y$ are Hausdorff topological vector spaces, $\mathrm X\subseteq\mathbf X$ and $\mathrm Y\subseteq \mathbf Y$ are nonempty, compact and convex sets, and $F,G:\mathrm X\strongly 2^{\mathrm Y}$ are two set-valued maps that satisfy:
 \begin{enumerate}
  \item For all $\mathrm x\in \mathrm X$, the set $F(\mathrm x)$ is open in $\mathrm Y$ and the set $G(\mathrm x)$ is not empty and convex.
  \item For all $\mathrm y\in\mathrm  Y$, the set $G^{-1}(\mathrm y)\defeq\{\mathrm x\in \mathrm X:\mathrm y\in G(\mathrm x)\}$ is open in $\mathrm X$ and the set $F^{-1}(\mathrm y)\defeq\{\mathrm x\in \mathrm X:y\in F(\mathrm x)\}$ is not empty and convex.
 \end{enumerate}
 Then there exists $\mathrm x_0\in \mathrm X$ such that $F(\mathrm x_0)\cap G(\mathrm x_0)\neq \emptyset$.
\end{proposition}
With the help of the latter and applying Proposition \ref{minimaxEq}, we will prove the existence of saddle points for $J$ in the next section.

\section{The main results}
\subsection{Existence of saddle points}
We begin the section with our main existence theorem and its proof.
\begin{theorem}\label{ExThm}
 The functional $J$, defined by \eqref{GeneralFunctional}, possesses a non-trivial saddle point
 \begin{equation*}
  ( u _0,\chi_0)\in H^1_{cyl}(D)\times  X (D),
 \end{equation*}
 that is, $u_0\nequiv 0$ and
 \begin{equation*}
  J( u _0,\chi_0)=\min_{ u \in H^1_{cyl}(D)}\max_{\chi\in  X (D)} J( u ,\chi).
 \end{equation*}
\end{theorem}
\begin{proof}
 We are looking for two pairs of functions $( u _{\mathrm{Mm}},\chi_{\mathrm{Mm}}),( u _{\mathrm{mM}},\chi_{\mathrm{mM}})\in H^1_{cyl}(D)\times  X (D)$, such that
 \begin{align}
  J( u _{\mathrm{Mm}},\chi_{\mathrm{Mm}}) {}& = \max_{\chi\in X(D)}\min_{ u \in H^1_{cyl}(D)} J( u ,\chi),\label{1n}\\[0.5ex]
  J( u _{\mathrm{mM}},\chi_{\mathrm{mM}}) {}& = \min_{ u \in H^1_{cyl}(D)}\max_{\chi\in X(D)} J( u ,\chi).\label{2n}
 \end{align}
 We first deal with the max-min case \eqref{1n}: Fix $\chi\in  X (D)$ and note that
 \begin{equation}
  \label{phiz} \int_D  u _z\;d\mathbf x\leq \int_D | u _z|\;d\mathbf x\leq \lnorm  u _z\rnorm_{L^2(D)} \lnorm 1 \rnorm_{L^2(D)}\leq \sqrt{|D|}\lnorm  u \rnorm_{cyl}.
 \end{equation}
 This, together with the growth condition on $M$ and \eqref{GeneralFunctional1} we calculate
 \begin{align}
  J( u ,\chi)\geq{}& \int_D \left(\chi\,\frac{1}{2}\,|\nabla u |^2+\frac{1}{2}(1-\chi)\,|\nabla u |^2\right) d\mathbf x-\mu\,\int_D  u _z\;d\mathbf x-J_2(\chi)\notag\\
  \geq{}& \frac{1}{2}\,\lnorm  u \rnorm_{cyl}^2-\mu\,\sqrt{|D|}\lnorm  u \rnorm_{cyl}-J_2(\chi),\label{coercivity}
 \end{align}
 which implies that the functional $J(\cdot,\chi)$ is bounded below and, in turn, the existence of a minimizing sequence $\{ u _{k,\chi}\}_{k\in\mathbb N}\subset H^1_{cyl}(D)$. Relation \eqref{coercivity} implies that this sequence is bounded and thus possesses a weak limit $ u _{\chi}\in H^1_{cyl}(D)$. Moreover, the functional $J(\cdot,\chi)$ is weakly lower semi-continuous: the real function $h:\mathbb R^3\times D\strongly \mathbb R$ with
 \begin{equation*}
  h(\xi,x,y,z)=  \chi(x,y,z)\,M(|\xi|)+\frac{1}{2}\Big(1-\chi(x,y,z)\Big)\,|\xi|^2-\mu\,\xi_3
 \end{equation*}
 is Carath\'eodory, satisfies $h(\xi,x,y,z)\geq \frac{1}{2}\,|\xi|^2-\mu\,\xi_3\in L^1(D)$ for almost all $\xi$, and it is strictly convex in $\xi$ a.e. in $D$ (since $M$ is convex and $\xi\mapsto |\xi|^2$ is strictly convex); see for example \cite[Theorem 1.6]{Struwe2008}. Moreover, the boundary integral term is weakly continuous, since the trace operator is compact from $H^1(D)$ into $L^2(\partial D)$ (see for example \cite[Theorem 6.2, p.103]{Necas2012}). Thus, $ u _{\chi}\in H^1_{cyl}(D)$ is a unique minimizer of $J(\cdot,\chi)$ in $H^1_{cyl}(D)$.
 
 The set
 \begin{equation}
  \label{SetMax} \mathbf X\defeq \big\{J( u _{\chi},\chi)=\min_{ u \in H^1_{cyl}(D)} J( u ,\chi):\chi \in X(D)\big\}
 \end{equation}
 is bounded above:
 \begin{equation}\label{UpBnd}
  J( u _{\chi},\chi)\leq J(0,\chi) = -J_2(\chi) \leq \frac{b}{2}\,|\Omega|.
 \end{equation}
 This implies that there exists a sequence $\{\chi_{k}\}_{k\in \mathbb N}\subseteq X(D)$ such that, if we denote $ u _k\defeq u _{\chi_k}$,
 \begin{equation*}
  \lim_{k\rightarrow\infty}J( u _{k},\chi_{k}) = \lim_{k\rightarrow\infty} \min_{ u \in H^1_{cyl}(D)} J( u ,\chi_{k})= \sup \mathbf X.
 \end{equation*}
 Suppose that $J_2(\chi_k) \to +\infty$. Then, we have by \eqref{UpBnd} $J( u _{k},\chi_k)\to -\infty$, a contradiction, since $\sup \mathbf X>-\infty$. Therefore, $J_2(\chi_k)$ is uniformly bounded, and therefore there exists a function $\chi_{\mathrm{Mm}}\in X(D)$ such that $\chi_k \wstar \chi_{\mathrm{Mm}}$ in $BV(D)$ up to a subsequence. We claim that the functional $J( u ,\cdot)$ is weakly-$*$ upper semi-continuous, so that
 \begin{equation*}
  \label{USC} J( u ,\chi_{\mathrm{Mm}})\geq \limsup_{k\rightarrow \infty} J( u ,\chi_{k})\geq \lim_{k\rightarrow \infty} \min_{ u \in H^1_{cyl}(D)} J( u ,\chi_{k})=\sup \mathbf X,
 \end{equation*}
 for all $ u \in H^1_{cyl}(D)$, and thus
 \begin{equation*}
  \min_{ u \in H^1_{cyl}(D)} J( u ,\chi_{\mathrm{Mm}})=J( u _{\chi_{\mathrm{Mm}}},\chi_{\mathrm{Mm}})\geq \sup \mathbf X.
 \end{equation*}
 However, $J( u _{\chi_{\mathrm{Mm}}},\chi_{\mathrm{Mm}})\in\mathbf X$ and thus $J( u _{\chi_{\mathrm{Mm}}},\chi_{\mathrm{Mm}})= \sup \mathbf X$ so that we have solved \eqref{1n} with $u_{\mathrm{Mm}}\defeq u _{\chi_{\mathrm{Mm}}}$.

 \textit{Proof of the claim.} Fix $ u \in H^1_{cyl}(D)$ and take a sequence $\{\chi_{k}\}_{k\in \mathbb N}\subset X(D)$ and $\chi\in X(D)$ such that $\chi_{k}\wstar \chi$. Since the total variation is weakly-$*$ lower semi-continuous in $BV(D)$ and $\chi_{k}\wstar \chi$ implies $\chi_{k}\strongly \chi$ in $L^1(D)$ we get that $J_2$ is lower semi-continuous. To finish the proof of the claim we need to prove that $J_1( u ,\cdot)$ is continuous with respect to the $L^1$-topology: consider the sequence of real numbers $\big\{J_1( u ,\chi_{k})\big\}_{k\in\mathbb N}$ and pick an arbitrary subsequence $\big\{J_1( u ,\chi_{k_l})\big\}_{l\in\mathbb N}$. For the sequence of functions $\{\chi_{k_l}\}_{l\in \mathbb N}$ it still holds that $\chi_{k_l}\strongly \chi$ in $L^1(D)$, and thus we can extract a subsequence $\{\chi_{k_{l_m}}\}_{m\in \mathbb N}$ converging to $\chi$ almost everywhere. Moreover, since $C_M>1$
 \begin{equation*}
  \chi_{k_{l_m}}\, M(|\nabla u|)+\frac{1}{2}\,(1-\chi_{k_{l_m}})\,|\nabla u|^2\leq \frac{1}{2}\,|\nabla u|^2+\frac{1}{2}\,\chi_{k_{l_m}}\,(C_M-1)\,|\nabla u|^2\leq \frac{C_M}{2}\,|\nabla u|^2
 \end{equation*}
 so that, using Lebesgue's dominated convergence theorem, we get
 \begin{equation*}
  J_1( u ,\chi_{k_{l_m}})\strongly J_1( u ,\chi).
 \end{equation*}
 All in all, we have shown that each subsequence $\big\{J_1( u ,\chi_{k_l})\big\}_{l\in\mathbb N}$ possesses a further subsequence that converges to the same limit, namely  to $J_1( u ,\chi)$. Thus
 \begin{equation*}
  J_1( u ,\chi_{n,k})\strongly J_1( u ,\chi_{n})
 \end{equation*}
 which finishes the proof of the claim.

 Next, we solve the min-max problem \eqref{2n} in a similar manner: using \eqref{phiz} and the growth condition of $M$ we get
 \begin{equation*}
  J(u,\chi) \leq \frac{C_M}{2}\lnorm u\rnorm_{cyl}^2+\mu\,\sqrt{|D|}\,\lnorm u\rnorm_{cyl}+\frac{b}{2}\,|\Omega|-\min\{p_0,\tau\}\,\lnorm \chi\rnorm_{BV},
 \end{equation*}
 since
 \begin{align*}
 J_2(\chi)= {}& b \int_\Omega\left(\int_{-1}^0 z\,\chi\; dz+\int_{0}^1 z\,\chi\; dz\right)dxdy+p_0\int_D|\chi|\; d\mathbf x+\tau\int_D |\nabla\chi|\\
 \geq{}&b \int_\Omega\int_{-1}^0 z\,\chi\; dzdxdy+p_0\int_D|\chi|\; d\mathbf x+\tau\int_D |\nabla\chi|\\
 \geq{}&-\frac{b}{2}\,|\Omega|+p_0\int_D|\chi|\; d\mathbf x+\tau\int_D |\nabla\chi|\\
 \geq{}&-\frac{b}{2}\,|\Omega|+\min\{p_0,\tau\}\,\left(\int_D|\chi|\; d\mathbf x+\int_D |\nabla\chi|\right),
\end{align*}
so that there exists a bounded maximizing sequence $\{\chi_{k,u}\}_{k\in\mathbb N}\subset X(D)$. Thus, there exists $\chi_{u}\in X(D)$ which, by the weak-$*$ upper semi-continuity of $J(u,\cdot)$, is a maximizer of $J(u,\cdot)$. The set 
 \begin{equation}
  \label{SetMin} \mathbf Y\defeq \big\{J(u,\chi_u)=\max_{ \chi \in X(D)} J( u ,\chi):u \in H^1_{cyl}(D)\big\}
 \end{equation}
 is bounded below:
 \begin{equation}\label{LowBnd}
   \begin{aligned}
    J(u,\chi_u)\geq {}& J(u,\chi_{\{z>0\}})\\
    = {}&\int_\Omega \int_0^1 M(|\nabla u |)\;dzdxdy+\frac{1}{2}\int_\Omega \int_{-1}^0 |\nabla u | ^2\;dzdxdy-\mu\int_D u_z\;d\mathbf x\\
    {}& -\int_\Omega\int_0^1 (b\,z+p_0)\;dzdxdy-\tau\,\int_D|\nabla\chi_{\{z>0\}}|\\
    \geq {}& \frac{1}{2}\, \lnorm u\rnorm_{cyl}^2-\mu \, \sqrt{|D|}\,\lnorm u\rnorm_{cyl}-(b+p_0+\tau)\,|\Omega|\\
    \geq {}& -(b+p_0+\tau+\mu^2)\,|\Omega|
   \end{aligned}
 \end{equation}
 for all $u\in H^1_{cyl}(D)$. Thus there exists a sequence $\{u_k\}_{k\in\mathbb N}$ such that
 \begin{equation*}
  \lim_{k\rightarrow\infty}J(u_{k},\chi_{k}) = \lim_{k\rightarrow\infty} \max_{ \chi \in X(D)} J( u_k ,\chi)= \inf \mathbf Y,
 \end{equation*}
 where again $\chi_k\defeq \chi_{u_k}$. Estimate \eqref{LowBnd} implies that $\{u_k\}_{k\in\mathbb N}$ is a bounded sequence in $H^1_{cyl}(D)$ and thus there exists $u_{\mathrm{mM}}\in H^1_{cyl}(D)$  such that $u_k\weakly u_{\mathrm{mM}}$. As already shown, $J(\cdot,\chi)$ is weakly lower semi-continuous so that
 \begin{equation*}
  \label{LSC} J( u_{mM} ,\chi)\leq \liminf_{k\rightarrow \infty} J( u_{k} ,\chi)\leq \lim_{k\rightarrow \infty} \max_{ \chi \in X(D)} J( u_k ,\chi)=\inf \mathbf Y,
 \end{equation*}
 for all $\chi\in X(D)$. Taking the maximum over $\chi$ we get $J( u_{mM} ,\chi_{u_{mM}})\leq\inf \mathbf Y$ and, since $J( u_{mM} ,\chi_{u_{mM}})\in \mathbf Y$, we get $J(u_{\mathrm{mM}},\chi_{\mathrm{mM}})=\inf \mathbf Y$ where $\chi_{\mathrm{mM}}\defeq \chi_{u_{\mathrm{mM}}}$.
 
 The last step is to prove that the functional $J$ has a saddle value, that is,
 \begin{equation*}
  J( u _{\mathrm{Mm}},\chi_{\mathrm{Mm}})=J( u _{\mathrm{mM}},\chi_{\mathrm{mM}})
 \end{equation*}
 To that end, first note that we directly obtain that
 \begin{equation*}
  \max_{\chi\in X(D)}\min_{ u \in H^1_{cyl}(D)} J( u ,\chi)\leq \min_{ u \in H^1_{cyl}(D)}\max_{\chi\in X(D)} J( u ,\chi),
 \end{equation*}
 that is,
 \begin{equation*}
 J( u _{\mathrm{Mm}},\chi_{\mathrm{Mm}}) \leq J( u _{\mathrm{mM}},\chi_{\mathrm{mM}}).
 \end{equation*}
 We will need the closed convex hull of $X(D)$, namely
 \begin{equation*}
  \clconv X(D)=\bigg\{\rho\in BV(D):0\leq \rho\leq 1\text{ a.e. in } D,\ \int_D \rho\;d\mathbf x=|\Omega|\bigg\}.
 \end{equation*}
 Note that $\clconv X(D)$ is a weakly-$*$ closed and convex subset of $BV(D)$ and that all partial semi-continuity properties of $J$ still hold in it. Assume there exists $c\in\mathbb R$ such that
 \begin{equation}\label{contra}
  \max_{\chi\in X(D)}\min_{ u \in H^1_{cyl}(D)} J( u ,\chi)<c< \min_{ u \in H^1_{cyl}(D)}\max_{\chi\in X(D)} J( u ,\chi),
 \end{equation}
 and define the set-valued maps $F,G:H^1_{cyl}(D)\strongly 2^{\clconv X(D)}$ by
 \begin{equation*}
  F( u )\defeq \big\{\rho\in \clconv X(D):J( u ,\rho)<c\big\}\ \text{ and }\  G( u )\defeq \big\{\rho\in \clconv X(D):J( u ,\rho)>c\big\}.
 \end{equation*}
 Since $J( u ,\cdot)$ is weakly-$*$ upper semi-continuous we get that $F( u )$ is open for each $ u \in H^1_{cyl}(D)$. For each $ u \in H^1_{cyl}(D)$ the set $G( u )$ is not empty, due to \eqref{contra}, and convex: let $\rho_1,\rho_2\in\clconv X(D)$ such that $J( u ,\rho_1),J( u ,\rho_2)>c$ and $t\in [0,1]$ and calculate
 \begin{equation}\label{convex}
  J\big( u ,t\,\rho_1+(1-t)\,\rho_2\big)\geq t\,J( u ,\rho_1)+(1-t)\,J( u ,\rho_2)>t\,c+(1-t)\,c=c,
 \end{equation}
 since the functional $J( u ,\cdot)$ is concave. Moreover, $G^{-1}(\rho)=\big\{ u \in H^1_{cyl}(D):J( u ,\rho)>c\big\}$ is open for each $\rho\in\clconv X(D)$, since $J(\cdot,\rho)$ is lower semi-continuous, and $F^{-1}(\rho)=\big\{ u \in H^1_{cyl}(D):J( u ,\rho)<c\big\}$ is not empty (due to \eqref{contra}) and convex for every $\rho\in\clconv X(D)$, since $J(\cdot,\rho)$ is convex (argue just like \eqref{convex}). Thus, applying Proposition \ref{KKM} ($BV(D)$ is isomorphic to the dual of a separable Banach space--see for example \cite[Remark 3.12, p. 124]{AmbrosioFuscoEtAl2000}--and the weak-$*$ topology is always Hausdorff in the dual of a Banach space) to obtain $(\tilde u,\tilde \rho)\in H^1_{cyl}(D)\times \clconv X(D)$ such that $\tilde \rho\in F(\tilde u)\cap G(\tilde u)$, i.e., $c<J(\tilde u,\tilde\rho)<c$, a contradiction.

 Thus, Proposition \ref{minimaxEq} implies that the functional $J$ possesses a saddle point $( u_0,\chi_0)\in H^1_{cyl}(D)\times X(D)$, i.e.,
 \begin{equation}
  \label{PhiRhoSaddle} J( u_0,\chi)\leq J( u_0,\chi_0)\leq J( u ,\chi_0)\ \text{ for all }\ ( u ,\chi)\in H^1_{cyl}(D)\times X(D),
 \end{equation}
 given by $( u_0,\chi_0)=( u _{\mathrm{mM}},\chi_{\mathrm{Mm}})$.
 
 Finally we illustrate the non-triviality of the saddle point: Let $\varphi\in C^\infty_0(\Omega)$ satisfy $\varphi\geq0$ and set $\tilde u(x,y,z)\defeq z\, \varphi(x,y)$, so that $\tilde u\in H^1_{cyl}(D)$. It holds that
$$\int_D\tilde u_z\;d\mathbf x=\int_D\varphi\;d\mathbf x=2\int_\Omega\varphi\;dxdy>0.$$
 For any $\chi\in X(D)$ and $\varepsilon>0$ we obtain using \eqref{assumeM} that
\begin{align*}
 J(\varepsilon\,\tilde u,\chi) \leq {}&\varepsilon^2 \int_D\left(\frac{C_M\,\chi}{2}|\nabla \tilde u|^2+\frac{1-\chi}{2}|\nabla \tilde u|^2\right)d\mathbf x-\varepsilon\,\mu\int_D\tilde u_z \;d\mathbf x-J_2(\chi)\\
 = {}&\varepsilon^2 \int_D\left(\frac{C_M\,\chi}{2}|\nabla \tilde u|^2+\frac{1-\chi}{2}|\nabla \tilde u|^2\right)d\mathbf x-\varepsilon\,\mu\int_\Omega \varphi \;dxdy+J(0,\chi).
\end{align*}
Thus, for $\varepsilon$ small we get $J(\varepsilon\,\tilde u,\chi)<J(0,\chi)$, in particular for $\chi=\chi_0$. 
\end{proof}

\subsection{Qualitative properties of the optimal configuration}
Let $(u_0,\chi_0)$ be a saddle point of $\mathcal{E}$. Define the set that the ferrofluid occupies by
\begin{equation}\label{D_F}
 D_F:= \{x \in D\,|\,\chi_0(x)=1\}.
\end{equation}
For all $(u,\chi)\in H_{cyl}^1(D)\times X(D)$ holds
$$J(u_0,\chi)\leq J(u,\chi_0).$$
Moreover
$$J(u_0,\chi)\geq \frac{1}{2}\int_D |\nabla u_0|^2\,d\mathbf x-\mu\,\int_D(u_0)_z\,d\mathbf x -J_2(\chi)$$
and
$$J(u,\chi_0)\leq \frac{C_M}{2}\int_D  |\nabla u|^2\,d\mathbf x-\mu\,\int_D(u)_z\,d\mathbf x -J_2(\chi_0),$$
so that altogether
\begin{equation}
 \label{eq8}\frac{1}{2}\,\big(\lnorm u_0\rnorm_{cyl}^2-C_M\,\lnorm u\rnorm_{cyl}^2\big)-\mu\,\int_D \big((u_0)_z-u_z\big)\,d\mathbf x\leq J_2(\chi)-J_2(\chi_0).
\end{equation}
This enables us to prove an estimate on the norm of the optimal solution.
\begin{proposition}\label{u0norm}
 It holds that $\lnorm u_0\rnorm_{cyl}\leq 2\,\mu\,\sqrt{|D|}$.
\end{proposition}
\begin{proof}
 For $u=0$ and $\chi=\chi_0$ in \eqref{eq8} we get that
 $$\frac{1}{2}\,\lnorm u_0\rnorm_{cyl}^2\leq \mu\,\int_D (u_0)_z\,d\mathbf x\leq \mu\,\sqrt{|D|}\,\lnorm u_0\rnorm_{cyl}.$$
 The last inequality is due to \eqref{phiz}.
\end{proof}
The gravity term in $J_2$ allows us to prove an estimate that justifies the physical intuition that a heavy ferrofluid (with $b$ large) will not float in the air.
\begin{proposition}\label{prop_bottom}
 Let $D_F$ be as in \eqref{D_F} and $\partial_{bot} D$, the bottom part of the boundary, as in \eqref{boundaries}. Then
 $$\dist(D_F,\partial_{bot}D)\leq \frac{\mu^2}{b}\left(1-\frac{1}{C_M}\right).$$ 
\end{proposition}
\begin{proof}
 If $d:=\dist(D_F,\partial_{bot}D)= 0$ there is nothing to prove. Suppose that $d >0$. For any $\delta \in [0,d)$, define $A_\delta\eqdef D_F-(0,0,\delta)$, and note that $\dist(A_\delta,\partial_{bot}D)>0$. Let $\alpha \in [0,C_M^{-\frac{1}{2}})$ (to be chosen appropriately) and set $u=\alpha\, u_0$ in inequality \eqref{eq8} to obtain
 {\allowdisplaybreaks\begin{align*}
  J_2(\chi)-J_2(\chi_0)\geq {}& \frac{1}{2} \big(1-C_M\,\alpha^2\big)\lnorm u_0\rnorm_{cyl}^2-\mu\,\big(1-\alpha\big)\int_D(u_0)_z\,d\mathbf x\\
  \geq {}& \frac{1}{2} \big(1-C_M\,\alpha^2\big)\lnorm u_0\rnorm_{cyl}^2-\mu\,\big(1-\alpha\big)\,\sqrt{|D|}\,\lnorm u_0\rnorm_{cyl}\\[1ex]
  \geq {}& -\frac{\mu^2\,|D|\,(1-\alpha)^2}{2(1-C_M\,\alpha^2)}
 \end{align*}}
 Setting $\chi=\chi_{A_\delta}$, the left-hand side becomes
 $$J_2(\chi_{A_\delta})-J_2(\chi_0)=b\int_D z\,(\chi_{A_\delta}-\chi_0)\,d\mathbf x,$$
 since a rigid motion of $D_F$ away from the boundary does not change its perimeter. It holds that
 $$\int_{D_F} z\,d\mathbf x=\int_{A_\delta} (z+\delta)\,d\mathbf x=\int_{A_\delta} z\,d\mathbf x+\delta\,|\Omega|,$$
 so that, altogether, we get
 $$-\frac{\mu^2\,(1-\alpha)^2}{(1-C_M\,\alpha^2)} \leq -b\,\delta,$$
 or, equivalently,
 \begin{equation*}
  \delta\leq \frac{\mu^2\,(1-\alpha)^2}{b(1-C_M\,\alpha^2)}. 
 \end{equation*}
 The function on the right hand side is minimized for $\alpha=\alpha_* \defeq C_M^{-1} < C_M^{-\frac{1}{2}}$. Choosing $\alpha=\alpha_*$ we obtain
 \begin{equation*}
  \label{eq4}\delta\leq \frac{\mu^2}{b}\left(1-\frac{1}{C_M}\right), 
 \end{equation*}
 and taking the supremum on all admissible $\delta$ we finish the proof.
%  \[ \dist(D_F,\partial_{bot}D) \leq \frac{\mu^2}{b}\left(1-\frac{1}{C_M}\right).\]
\end{proof}
\begin{remark}
The same argumentation provides with a proof that there will be no disconnected ferrofluid bubbles floating far from the rest of the ferrofluid mass. A complementary argument can be used to show that there cannot be any air bubbles in the ferrofluid too close to the interface.
\end{remark}
\begin{remark}
One could have chosen to eliminate the quadratic part of \eqref{u0norm} by choosing $u=C_M^{-1/2}\,u_0$ to obtain a bound directly. The argumentation in the proof above aims to illustrate that the result of Proposition \ref{prop_bottom} is optimal, in the sense that a different rescaling of the optimal solution will not provide with a better estimate. This, of course, does not mean that the proposition is optimal per se: We expect that an explicit relation between the parameters exists, that acts as a necessary and sufficient condition for asserting that $\dist(D_F,\partial_{bot}D)=0$. This follows physical intuition, since a light ferrofluid in a strong magnetic field will completely leave the bottom and stick to the upper part of the container. However, finding such a relation needs sharper density estimates for the optimal solution, whose extraction and manipulation lies beyond the scope of this paper.
\end{remark}

\subsection{Duality theory for linear magnetization laws}\label{duality}
In this section we focus in the linear case, that is, we assume that 
$$M(s)=\frac{\mu}{2}\, s^2$$
for a fixed $\mu>1$. We prove that the saddle point that we found is a minimizer of the energy functional of the system. In fact, we show that the energy functional is conjugate to our convex-concave functional $J$. The main impact of this section is that, after obtaining Theorem \ref{EquivThm}, we can apply regularity results that have been developed for minimizers of free discontinuity problems to the saddle point from Theorem \ref{ExThm}, in order to obtain Theorem \ref{LinRegularity} and Corollary \ref{LinRegularity1}.

Among the first works studying properties of optimal configurations that are minimizers of a corresponding energy we should mention those by Ambrosio and Buttazzo \cite{AmbrosioButtazzo1993} and by Lin \cite{Lin1993}. Apart from other technical results, Ambrosio and Buttazzo also proved H\"older continuity of the optimal solution and openness of the optimal set, whereas Lin worked in a space of currents. There has been a number of works following; Larsen \cite{Larsen2003} showed $C^1$ regularity away from the boundary for the components of the optimal set; Fusco and Julin \cite{FuscoJulin2015} dealt with the so-called \textit{Taylor cones} -- conical points on the free surface of a fluid inside an electric field -- and with refined regularity assertions on the minimizers; De Philippis and Figalli \cite{DePhilippisFigalli2015} studied the dimension of the set of singularities of the boundary of the optimal set; Carozza \textit{et al} \cite{CarozzaEtAl2014} deal also with energies with general potentials. Many other works on this subject can be found in the references in these citations; this list is by far not exhaustive.

The main result of this section is given in its end, following the necessary discussion on duality theory. We follow the notation and exposition of Ekeland-Temam \cite[Chapter III, 4]{EkelandTemam1999}.

Fix $\chi \in X(D)$ and define $V\defeq H^1_{cyl}(D)$ and $Y\defeq L^2(D,\mathbb R^3)\cong Y^*$. Moreover, define
\begin{equation}
 \label{NewJ}
  J_\chi(\nabla u)  \defeq J(u,\chi)= \int_D f_\chi(\nabla u)\; d\mathbf x-J_2(\chi),
\end{equation}
where $f_\chi:D\times\mathbb R^3\strongly \mathbb R$ defined by
\begin{equation}
 \label{integrant}
  f_\chi(\xi)\defeq \frac{\mu\,\chi}{2}\,|\xi |^2+\frac{1-\chi}{2}\,|\xi |^2-(0,0,\mu)\cdot \xi.
\end{equation}
Moreover, define the family of perturbations $\Phi_\chi:V\times Y\strongly \mathbb R$ by
\begin{equation*}
 \Phi_\chi(u,p)\defeq J_\chi(\nabla u-p).
\end{equation*}
In the previous section we have shown that there exists a unique solution $u_\chi\in V$ to the minimization problem
\begin{equation}
 \label{P} \min_{u\in V} J_\chi(\nabla u)=\min_{u\in V} \Phi_\chi(u,0).
\end{equation}
The dual problem is
\begin{equation}
 \label{P*pre} \sup_{p^*\in Y} \big\{-\Phi^*_\chi(0,p^*)\big\}
\end{equation}
and it holds that
\begin{align*}
 \Phi^*_\chi(0,p^*) {}& =\sup\bigg\{\int_D p^*\cdot p\;d\mathbf x-J_\chi(\nabla u- p):u\in V,p\in Y\bigg\}\\
 {}&= \sup\left\{\sup\bigg\{\int_D p^*\cdot p\;d\mathbf x-J_\chi(\nabla u- p):p\in Y\bigg\}:u\in V\right\}\\
 {}&= \sup\left\{\sup\bigg\{\int_Dp^*\cdot\nabla u\;d\mathbf x-\int_D p^*\cdot q\;d\mathbf x-J_\chi(q):\nabla u-q\in Y\bigg\}:u\in V\right\}\\
 {}&= \sup\left\{\sup\bigg\{\int_Dp^*\cdot\nabla u\;d\mathbf x-\int_D p^*\cdot q\;d\mathbf x-J_\chi(q):q\in Y\bigg\}:u\in V\right\}\\
 {}&=\left\{
 \begin{alignedat}{4}
  &\sup\bigg\{-\int_D p^*\cdot q\;d\mathbf x-J_\chi(q):q\in Y\bigg\},{}&& \text{ when } p^*\in Y_d,\\
  &+\infty, {}&& \text{ otherwise }
 \end{alignedat}
 \right.\\[0.5em]
 {}&=\left\{
 \begin{alignedat}{4}
  &J^*_\chi(-p^*),{}&& \text{ when } p^*\in Y_d,\\
  &+\infty, {}&& \text{ otherwise, }
 \end{alignedat}
 \right.
\end{align*}
where
$$Y_d\defeq \left\{p^*\in Y:\int_D p^*\cdot\nabla u\;d\mathbf x=0\text{ for all }u\in V\right\}.$$
Thus, \eqref{P*pre} is equivalent to
\begin{equation}
 \label{P*pre1} \sup_{p^*\in Y_d}\left\{-J^*_\chi(-p^*)\right\}=-\inf_{p^*\in Y_d} J^*_\chi(-p^*).
\end{equation}
Next, we calculate the conjugate functional with the help of \cite[Proposition 1.2, p.78]{EkelandTemam1999}
\begin{equation}
 \label{dual1} J^*_\chi(p^*)=\int_D f^*_\chi(p^*)\;d\mathbf x+J_2(\chi),
\end{equation}
since $J_2(\chi)$ is a constant and $f^*_\chi(p^*)$ is calculated pointwise in $D$, so that,
\begin{equation*}
\begin{aligned}
 f^*_\chi(p^*){}&=\left\{\
 \begin{aligned}
  \left(\xi\mapsto \frac{1}{2}|\xi|^2-(0,0,\mu)\cdot \xi\right)^*(p^*) &\quad \text{ for } \chi=0,\\
  \left(\xi\mapsto \frac{\mu}{2}|\xi|^2-(0,0,\mu)\cdot \xi\right)^*(p^*) &\quad \text{ for }\chi=1,
 \end{aligned}
 \right.
 \\[0.5em]
 {}&=\frac{\chi}{2\,\mu}\big|p^*+(0,0,\mu)\big|^2+\frac{1-\chi}{2}\big|p^*+(0,0,\mu)\big|^2,
 \end{aligned}
\end{equation*}
where we used classical properties of the convex conjugate (ex. \cite[Proposition 1.3.1, p.42]{HiriartUrrutyLemarechal1993}). Thus, the dual problem \eqref{P*pre1} becomes
\begin{equation}
 \label{P*} -\inf_{p^*\in Y_d} \tilde{\mathcal E}_\chi(p^*),
\end{equation}
where
\begin{equation}
 \label{FunctionalE} \tilde{\mathcal E}_\chi(p^*)\defeq \int_D\left\{\frac{\chi}{2\,\mu}\,\big|p^*-(0,0,\mu)\big|^2+\frac{1-\chi}{2}\,\big|p^*-(0,0,\mu)\big|^2\right\}d\mathbf x+J_2(\chi).
\end{equation}
Next, we calculate the derivative of the perturbations
\begin{align*}
 &\langle \Phi_\chi'(u,p),(v,q)\rangle\\
 &=\int_D\Big\{\chi\,\mu\,(\nabla u-p)\cdot (\nabla v-q)+(1-\chi)\,(\nabla u-p)\cdot(\nabla v-q)-(0,0,\mu)\cdot (\nabla v-q)\Big\}\;d\mathbf x,
\end{align*}
so that the differential satisfies
\begin{equation*}
\mathrm d \Phi_\chi(u,p)=\left(
\begin{aligned}
 \chi\,\mu\,(\nabla u-p)+(1-\chi)\,(\nabla u-p)-(0,0,\mu)\\
 -\chi\,\mu\,(\nabla u-p)-(1-\chi)\,(\nabla u-p)+(0,0,\mu)
\end{aligned}
\right)^\top\in Y\times Y\subset V^*\times Y.
\end{equation*}
According to \cite[Proposition 5.1, p.21]{EkelandTemam1999},
\begin{equation*}
 \Phi_\chi(u,0)+\Phi^*_\chi(0,p^*)=0
\end{equation*}
is equivalent to
\begin{equation}\label{subdiff}
 (0,p^*)\in \partial\Phi_\chi(u,0),
\end{equation}
the latter denoting the subdifferential of $\Phi_\chi$. But $\Phi_\chi$ is differentiable, so that $\partial\Phi_\chi (u,0)=\{\mathrm d\Phi_\chi (u,0)\}$. Thus, \eqref{subdiff} is equivalent to
\begin{alignat}{3}
 0{}&=\chi\,\mu\,\nabla u+(1-\chi)\,\nabla u-(0,0,\mu) &\quad \text{(in the sense of distributions)} \label{eq1}\\
 p^*{}&=-\chi\,\mu\,\nabla u-(1-\chi)\,\nabla u+(0,0,\mu), &\label{eq2}
\end{alignat}
where \eqref{eq1} translates to
\begin{equation}
 \label{eq3}\int_D\Big(\chi\,\mu\,\nabla u+(1-\chi)\,\nabla u-(0,0,\mu)\Big)\cdot\nabla v\;d\mathbf x=0\text{ for all } v\in V.
\end{equation}
A solution to the equation \eqref{eq3} is $u=u_\chi$, the minimizer of the primal problem \eqref{P}. Set 
$$p_\chi^*\defeq -\chi\,\mu\,\nabla u_\chi-(1-\chi)\,\nabla u_\chi+(0,0,\mu).$$
Then, from \cite[Proposition 2.4, p.53]{EkelandTemam1999} we get that $p_\chi^*$ is a solution to the dual problem \eqref{P*}. Define $\mathcal E_\chi:Y\strongly\mathbb R$ by
\begin{equation}
 \label{energy}
 \mathcal E_\chi(q)\defeq \int_D\left(\frac{\chi\,\mu}{2}\,|q|^2+\frac{1-\chi}{2}\,|q|^2\right)d\mathbf x+J_2(\chi).
\end{equation}
We have the following:
\begin{lemma}
 \label{dual-new} The function $u_\chi$ satisfies
 \begin{equation}\label{eq10}
  \mathcal E_\chi(\nabla u_\chi)=\min\left\{\mathcal E_\chi(\nabla u): u\in V_\chi\right\}=\min_{p^*\in Y_d}\tilde{\mathcal E}_\chi(p^*),
 \end{equation}
 where the space $V_\chi$ is defined by
\begin{equation}
 \label{Space}V_\chi\defeq \left\{v\in V:\int_D \big(-\chi\,\mu\,\nabla v-(1-\chi)\,\nabla v+(0,0,\mu)\big)\cdot\nabla u\;d\mathbf x=0\text{ for all } u\in V\right\}.
\end{equation}
In particular, it holds
\begin{equation}
  \mathcal E_\chi(\nabla u_\chi)=\min\left\{\mathcal E_\chi(\nabla u): u\in V_\chi\text{ and } u =u_\chi\text{ on }\partial D\right\}.
 \end{equation}
\end{lemma}
\begin{proof}
First note that
\begin{align*}
 {}&\min_{p^*\in Y_d}\tilde{\mathcal E}_\chi(p^*) = \tilde{\mathcal E}_\chi(p_\chi^*)\\
 {}&\quad=\int_D\left(\frac{\chi}{2\,\mu}\,\big|\chi\,\mu\,\nabla u_\chi+(1-\chi)\,\nabla u_\chi\big|^2+\frac{1-\chi}{2}\,\big|\chi\,\mu\,\nabla u_\chi+(1-\chi)\,\nabla u_\chi\big|^2\right)d\mathbf x+J_2(\chi)\\
 {}&\quad =\int_D\left(\frac{\chi\,\mu}{2}\,|\nabla u_\chi|^2+\frac{1-\chi}{2}\,|\nabla u_\chi|^2\right)d\mathbf x+J_2(\chi).
\end{align*}
Moreover
\begin{align*}
\mathcal E_\chi(\nabla u_\chi)={}&\min\Big\{\tilde{\mathcal E}_\chi(p^*):p^*\in Y_d\Big\}\\
={}&\min\bigg\{\tilde{\mathcal E}_\chi(p^*):p^*\in Y\text{ and }\int_D p^*\cdot\nabla u\;d\mathbf x=0\text{ for all }u\in V\bigg\}\\
={}&\min\bigg\{\tilde{\mathcal E}_\chi\big(-\chi\,\mu\,q-(1-\chi)\,q+(0,0,\mu)\big):q\in Y\text{ and}\\
{}&\qquad \int_D \big(-\chi\,\mu\,q-(1-\chi)\,q+(0,0,\mu)\big)\cdot\nabla u\;d\mathbf x=0\text{ for all } u\in V\bigg\}\\
={}&\min\bigg\{\mathcal E_\chi(q):q\in Y\text{ and}\\
{}&\qquad \int_D \big(-\chi\,\mu\,q-(1-\chi)\,q+(0,0,\mu)\big)\cdot\nabla u\;d\mathbf x=0\text{ for all } u\in V\bigg\}\\
\leq{}&\inf\bigg\{\mathcal E_\chi(\nabla v):v\in V_\chi\bigg\}.
\end{align*}
Since from its definition $u_\chi\in V_\chi$, we obtain that
$$\mathcal E_\chi(\nabla u_\chi)=\min_{u\in V_\chi}\mathcal E_\chi(\nabla u).$$
% Since the non-constant summand of $\mathcal E_\chi$ is nonnegative we get that
% $$u_\chi=\argmin\bigg\{\frac{3}{2}\int_D\left\{\chi\,\mu\,|\nabla v|^2+(1-\chi)\,|\nabla v|^2\right\}d\mathbf x+J_2(\chi):v\in V_\chi\bigg\}.$$
% Adding a zero results in
% $$u_\chi=\argmin\bigg\{\mathcal E_\chi(\nabla v)+\int_D(0,0,\mu)\cdot \nabla v\;d\mathbf x:v\in V_\chi\bigg\}.$$
% Since
% \begin{equation}\label{eq7}
%  \int_D (0,0,\mu)\cdot \nabla v\;d\mathbf x = \int_D v_z\;d\mathbf x = \int_\Omega \left(v|_{z=1}-v|_{z=-1}\right)dxdy,
% \end{equation}
Since the minimization problem does not change when we consider it in the class of functions that satisfy the ``correct'' (i.e., $u=u_\chi$ on $\partial D$) boundary condition, we obtain
$$u_\chi=\argmin\bigg\{\mathcal E_\chi(\nabla v):v\in V_\chi\text{ and }v=u_\chi\text{ on }\partial D\bigg\}.$$
\end{proof}
We can now prove the following:
\begin{theorem}\label{EquivThm}
 Define the energy functional
 $$\mathcal E(u,\chi)\defeq \mathcal E_\chi(\nabla u)-\int_D(0,0,\mu)\cdot\nabla u\;d\mathbf x,$$
 where $\mathcal E_\chi$ is given by \eqref{energy}, and let $( u _0,\chi_0)\in H^1_{cyl}(D)\times  X (D)$ satisfy
 \begin{equation*}
  J( u _0,\chi_0)=\min_{ u \in H^1_{cyl}(D)}\max_{\chi\in  X (D)} J( u ,\chi).
 \end{equation*}
 Then $( u _0,\chi_0)$ satisfies
 \begin{equation*}
  \mathcal E(u_0,\chi_0)=\min\bigg\{\mathcal E(u,\chi):(u,\chi)\in H^1_{cyl}(D)\times  X (D)\text{ with } u=u_0\text{ on }\partial D\bigg\}.
 \end{equation*}
\end{theorem}
\begin{proof}
 From the discussion above we get that for $\chi\in X(D)$, the dual problem \eqref{P*} is equivalent to the problem
 \begin{equation}
  \max\bigg\{-\mathcal E_\chi(\nabla v):v\in V_\chi\text{ and } v =u_\chi\text{ on }\partial D\bigg\},
 \end{equation}
 where $V_\chi$ is defined in \eqref{Space}. From \cite[Proposition 2.4, p.53]{EkelandTemam1999} we get that
 \begin{equation}\label{eq5}
  \max\bigg\{-\mathcal E_\chi(\nabla v):v\in V_\chi\text{ and } v =u_\chi\text{ on }\partial D\bigg\}=\min\bigg\{J( v,\chi):v\in H^1_{cyl}(D)\bigg\}
 \end{equation}
 and, from Lemma \ref{dual-new}, that
 \begin{equation}\label{eq6}
  u_\chi=\argmin\bigg\{\mathcal E_\chi(\nabla v):v\in V_\chi\text{ and } v =u_\chi\text{ on }\partial D\bigg\}
 \end{equation}
 for all $\chi\in X(D)$. Since the function
 $$\chi\mapsto \min\bigg\{J( v,\chi):v\in H^1_{cyl}(D)\bigg\}$$
 (minimizers are unique so the mapping is well-defined as a real function) is maximized for $\chi=\chi_0$, we get that $\chi_0$ maximizes
 \begin{align*}
  \chi\mapsto {}& \max\bigg\{-\mathcal E_\chi(\nabla v):v\in V_\chi\text{ and } v =u_\chi\text{ on }\partial D\bigg\}\\
  {}& =-\min\bigg\{\mathcal E_\chi(\nabla v):v\in V_\chi\text{ and } v =u_\chi\text{ on }\partial D\bigg\}\\
  {}& =-\mathcal E_\chi(u_\chi),
 \end{align*}
 where, the last equality is due to \eqref{eq6}. Using equation \eqref{eq6} again, we get that
 $$\mathcal E_{\chi_0}(\nabla u_0)=\min\bigg\{\mathcal E_\chi(\nabla u):(u,\chi)\in V_0\times  X (D)\text{ with } u=u_0\text{ on }\partial D\bigg\},$$
 where $V_0\defeq V_{\chi_0}$. Because of 
 \begin{equation*}
 \int_D (0,0,\mu)\cdot \nabla u_0\;d\mathbf x = \int_D (u_0)_z\;d\mathbf x = \int_\Omega \left(u_0|_{z=1}-u_0|_{z=-1}\right)dxdy,
\end{equation*}
we can add the missing term on both sides to obtain
 \begin{align*}
  \mathcal E_{\chi_0}(u_0){}&-\int_D(0,0,\mu)\cdot\nabla u_0\;d\mathbf x\\
  {}&=\min\bigg\{\mathcal E_\chi(\nabla u)-\int_D(0,0,\mu)\cdot\nabla u_0\;d\mathbf x:(u,\chi)\in V_0\times  X (D)\text{ with } u=u_0\text{ on }\partial D\bigg\}\\
  {}&=\min\bigg\{\mathcal E_\chi(\nabla u)-\int_D(0,0,\mu)\cdot\nabla u\;d\mathbf x:(u,\chi)\in V_0\times  X (D)\text{ with } u=u_0\text{ on }\partial D\bigg\},
 \end{align*}
 that is,
 \begin{equation}\label{eq9}
 \mathcal E(u_0,\chi_0)=\min\bigg\{\mathcal E(u,\chi):(u,\chi)\in V_0\times  X (D)\text{ with } u=u_0\text{ on }\partial D\bigg\}.
 \end{equation}
 In order to finish the proof, note that the minimizer of $\mathcal E$ in $H^1_{cyl}(D)\times  X (D)$ belongs to $V_0\times  X (D)$, since the side condition in $V_0$ is nothing else than the partial Euler-Lagrange equation of $\mathcal E$.
\end{proof}
% \begin{remark}
%  Using a Lagrange multiplier in order to drop the side constraint that appears in $V_\chi$ (or $Y_d$) in \eqref{eq10} takes us back to the original saddle-point formulation for $J$. This is the reason for inserting the boundary condition $u=u_0$ on $\partial D$ and add the missing boundary term as a constant.
% \end{remark}

Since $(u_0,\chi_0)$ minimizes an energy functional, it is possible to apply the theory developed in the references to obtain regularity results. More precisely, we have the following proposition.

\begin{theorem}\label{LinRegularity}
Let $(u_0,\chi_0)$ be a minimizer of $\mathcal{E}$ and $D_F$ be given by \eqref{D_F}, that is, the set occupied by the ferrofluid. Then $u \in C^{0,\frac{1}{2}}_{loc}(D)$, and $\partial D_F\cap D$ is locally a $C^{1,\alpha}$-submanifold of $\mathbb R^3$ for some $\alpha \in( 0,1)$, that is, up to a relatively closed singular set $\Sigma$ which satisfies $\mathcal{H}^{2}(\Sigma)=0$. Here $\mathcal H^2$ denotes the $2$-dimensional Hausdorff measure in $\mathbb R^3$, restricted on $\partial D_F\cap D$.
\end{theorem}
\begin{proof}
Because of Theorem \ref{EquivThm}, The regularity results of \cite[Theorem 1.1]{LinKohn1999} and \cite[Theorem 1.2]{LinKohn1999} apply and provide the claim. In the notation of that paper, we need to set $F(x,u,p)=\frac{1}{2}|p|^2$, and $G(x,u,p)=\frac{\mu-1}{2} |p|^2 + bz + p_0$ (note that $\mu>1$ by assumption).
\end{proof}
As a direct consequence of the above we obtain that the optimal set is equivalent to a relatively open set.
\begin{corollary}\label{LinRegularity1}
 Let $(u_0,\chi_0)$ be a minimizer of $\mathcal{E}$ and $D_F$ as in Theorem \ref{LinRegularity}. Let $\tilde \chi$ be the characteristic function of the set $D_F\setminus(\partial D_F\cap D)$. Then $\mathcal E(u_0,\tilde \chi)\leq \mathcal E(u_0,\chi_0)$.
\end{corollary}
\begin{proof}
 Theorem \ref{LinRegularity} implies that $\mathcal H^3(\partial D_F\cap D)=0$ since the set $(\partial D_F\cap D)\setminus \Sigma$ is a $2$-dimensional submanifold. So we get from the definition of the perimeter that $\int_D|\nabla \tilde \chi|= \int_D|\nabla\chi_0|$ which implies the corollary.
\end{proof}

Thus in the linear case one can obtain a solution as a minimizer instead of a saddle point. That, in turn, allows for the application of the deep theory which was developed in the references listed in the beginning of Section \ref{duality} for minimizers of free discontinuity problems.

% \bibliographystyle{plain}
% \bibliography{references}
\input{Parini_Stylianou_2017_arXiv.bbl.references}

\end{document}

%% file: preamble.tex
\usepackage[bbgreekl]{mathbbol}
\usepackage{amsfonts}
\usepackage{amsmath}
\usepackage{amssymb}
\DeclareSymbolFontAlphabet{\mathbb}{AMSb}
\DeclareSymbolFontAlphabet{\mathbbl}{bbold}

\usepackage[greek,english,ngerman]{babel}
\usepackage{graphicx}
\usepackage{amsmath}
\usepackage{verbatim}
\usepackage{color}
\usepackage{amstext}
\usepackage{latexsym}
\usepackage[small,it]{caption}
\usepackage{dsfont}
\usepackage{enumerate}
\usepackage{mathtools}
\usepackage{indentfirst}
\usepackage{subcaption}
\usepackage{eurosym}
\usepackage{stmaryrd}

\usepackage[h]{esvect} % long vectors over things

\usepackage{bm}
\usepackage{amsthm}

\usepackage{url}

\makeatletter
\def\@endtheorem{\endtrivlist}% NEW
\makeatother
\theoremstyle{plain}
\newtheorem*{theorem*}{\protect\TheoremName}
\newtheorem*{acknowledgement*}{\protect\AcknowledgmentName}
\newtheorem*{algorithm*}{\protect\AlgorithmName}
\newtheorem*{assumption*}{\protect\AssumptionName}
\newtheorem*{assumptions*}{\protect\AssumptionsName}
\newtheorem*{axiom*}{\protect\AxiomName}
\newtheorem*{condition*}{\protect\ConditionName}
\newtheorem*{conditions*}{\protect\ConditionsName}
\newtheorem*{case*}{\protect\CaseName}
\newtheorem*{claim*}{\protect\ClaimName}
\newtheorem*{conclusion*}{\protect\ConclusionName}
\newtheorem*{conjecture*}{\protect\ConjectureName}
\newtheorem*{corollary*}{\protect\CorollaryName}
\newtheorem*{criterion*}{\protect\CriterionName}
\newtheorem*{definition*}{\protect\DefinitionName}
\newtheorem*{lemma*}{\protect\LemmaName}
\newtheorem*{notation*}{\protect\NotationName}
\newtheorem*{problem*}{\protect\ProblemName}
\newtheorem*{proposition*}{\protect\PropositionName}
\newtheorem*{solution*}{\protect\SolutionName}
\newtheorem*{summary*}{\protect\SummaryName}

\newtheorem{theorem}{\protect\TheoremName}[section]

\newtheorem{assumptions}[theorem]{\protect\AssumptionsName}

\newtheorem{corollary}[theorem]{\protect\CorollaryName}

\newtheorem{definition}[theorem]{\protect\DefinitionName}
\newtheorem{lemma}[theorem]{\protect\LemmaName}

\newtheorem{proposition}[theorem]{\protect\PropositionName}

\theoremstyle{definition}
\newtheorem*{example*}{\protect\ExampleName}
\newtheorem*{exercise*}{\protect\ExerciseName}
\newtheorem*{remark*}{\protect\RemarkName}

\newtheorem{remark}[theorem]{\protect\RemarkName}

\newcommand{\TheoremName}{}
\newcommand{\AcknowledgmentName}{}
\newcommand{\AlgorithmName}{}
\newcommand{\AssumptionName}{}
\newcommand{\AssumptionsName}{}
\newcommand{\AxiomName}{}
\newcommand{\ConditionName}{}
\newcommand{\ConditionsName}{}
\newcommand{\CaseName}{}
\newcommand{\ClaimName}{}
\newcommand{\ConjectureName}{}
\newcommand{\ConclusionName}{}
\newcommand{\CorollaryName}{}
\newcommand{\CriterionName}{}
\newcommand{\DefinitionName}{}
\newcommand{\LemmaName}{}
\newcommand{\NotationName}{}
\newcommand{\ProblemName}{}
\newcommand{\PropositionName}{}
\newcommand{\SolutionName}{}
\newcommand{\SummaryName}{}
\newcommand{\ExampleName}{}
\newcommand{\RemarkName}{}
\newcommand{\ExerciseName}{}
\addto\captionsenglish{%
  \renewcommand{\TheoremName}{Theorem}
  \renewcommand{\AcknowledgmentName}{Acknowledgments}
  \renewcommand{\AlgorithmName}{Algorithm}
  \renewcommand{\AssumptionName}{Assumption}
  \renewcommand{\AssumptionsName}{Assumptions}
  \renewcommand{\AxiomName}{Axiom}
  \renewcommand{\ConditionName}{Condition}
  \renewcommand{\ConditionsName}{Conditions}
  \renewcommand{\CaseName}{Case}
  \renewcommand{\ClaimName}{Claim}
  \renewcommand{\ConjectureName}{Conjecture}
  \renewcommand{\ConclusionName}{Conclusion}
  \renewcommand{\CorollaryName}{Corollary}
  \renewcommand{\CriterionName}{Criterion}
  \renewcommand{\DefinitionName}{Definition}
  \renewcommand{\LemmaName}{Lemma}
  \renewcommand{\NotationName}{Notation}
  \renewcommand{\ProblemName}{Problem}
  \renewcommand{\PropositionName}{Proposition}
  \renewcommand{\SolutionName}{Solution}
  \renewcommand{\SummaryName}{Summary}
  \renewcommand{\ExampleName}{Example}%
  \renewcommand{\ExerciseName}{Exercise}%
  \renewcommand{\RemarkName}{Remark}%
}
\addto\captionsngerman{%
  \renewcommand{\TheoremName}{Theorem}
  \renewcommand{\AcknowledgmentName}{Anerkennungen}
  \renewcommand{\AlgorithmName}{Algorithmus}
  \renewcommand{\AssumptionName}{Voraussetzung}
  \renewcommand{\AssumptionsName}{Voraussetzungen}
  \renewcommand{\AxiomName}{Axiom}
  \renewcommand{\ConditionName}{Bedingung}
  \renewcommand{\ConditionsName}{Bedingungen}
  \renewcommand{\CaseName}{Fall}
  \renewcommand{\ClaimName}{Behauptung}
  \renewcommand{\ConjectureName}{Vermutung}
  \renewcommand{\ConclusionName}{Folgerung}
  \renewcommand{\CorollaryName}{Korollar}
  \renewcommand{\CriterionName}{Kriterium}
  \renewcommand{\DefinitionName}{Definition}
  \renewcommand{\LemmaName}{Lemma}
  \renewcommand{\NotationName}{Notation}
  \renewcommand{\ProblemName}{Problem}
  \renewcommand{\PropositionName}{Satz}
  \renewcommand{\SolutionName}{L\"osung}
  \renewcommand{\SummaryName}{Zusammenfassung}
  \renewcommand{\ExampleName}{Beispiel}%
  \renewcommand{\ExerciseName}{\"Ubung}%
  \renewcommand{\RemarkName}{Bemerkung}%
}
\addto\captionsgreek{%
  \renewcommand{\TheoremName}{Je\'wrhma}
  \renewcommand{\AcknowledgmentName}{Euqarist\'ies}
  \renewcommand{\AlgorithmName}{Alg\'orijmos}
  \renewcommand{\AssumptionName}{Up\'ojesh}
  \renewcommand{\AssumptionsName}{Upoj\'eseis}
  \renewcommand{\AxiomName}{Ax\'iwma}
  \renewcommand{\ConditionName}{Pro\"up\'ojesh}
  \renewcommand{\ConditionsName}{Pro\"upoj\'eseis}
  \renewcommand{\CaseName}{Per\'iptwsh}
  \renewcommand{\ClaimName}{Isqurism\'os}
  \renewcommand{\ConjectureName}{Eikas\'ia}
  \renewcommand{\ConclusionName}{Sump\'erasma}
  \renewcommand{\CorollaryName}{P\'orisma}
  \renewcommand{\CriterionName}{Krit\'hrio}
  \renewcommand{\DefinitionName}{Orism\'os}
  \renewcommand{\LemmaName}{L\'hmma}
  \renewcommand{\NotationName}{Shmeiograf\'ia}
  \renewcommand{\ProblemName}{Pr\'oblhma}
  \renewcommand{\PropositionName}{Pr\'otash}
  \renewcommand{\SolutionName}{L\'ush}
  \renewcommand{\SummaryName}{Per\'ilhyh}
  \renewcommand{\ExampleName}{Par\'adeigma}%
  \renewcommand{\ExerciseName}{\'Askhsh}%
  \renewcommand{\RemarkName}{Parat\'hrhsh}%
}
\makeatletter
\renewenvironment{proof}[1][\proofname]{\par
\pushQED{\hfill$\blacksquare$}%
\normalfont \topsep6\p@\@plus6\p@\relax
\trivlist
\item\relax
{\bfseries
#1\@addpunct{.}}\hspace\labelsep\ignorespaces
}{%
\popQED\endtrivlist\@endpefalse
}
\makeatother

\newcommand{\strongly}{%
  \mathrel{
    \resizebox{1.4em}{0.88ex}{$\rightarrow$}
}}

\newcommand{\weakly}{%
  \mathrel{
    \resizebox{1.4em}{0.88ex}{$\rightharpoonup$}
}}

\newcommand{\wstar}{%
  \mathrel{\vbox{\offinterlineskip\ialign{%
    \hfil##\hfil\cr
    $\scriptstyle*\,$\cr
    \noalign{\kern 0.12ex}
    \resizebox{1.4em}{0.88ex}{$\rightharpoonup$}\cr
}}}}

\def\XXint#1#2#3{{\setbox0=\hbox{$#1{#2#3}{\int}$ }\hspace{0.17em}
\vcenter{\hbox{$#2#3$ }}\kern-.585\wd0}}

\DeclareMathOperator*{\argmin}{arg\,min}

\DeclareMathOperator*{\clconv}{\overline{conv}}
\DeclareMathOperator{\nequiv}{\not\equiv}

\DeclareMathOperator{\diver}{div}

\DeclareMathOperator{\dist}{dist}

\newcommand{\lnorm}{\left\|}
\newcommand{\rnorm}{\right\|}


\makeatletter
\newcommand{\defeq}{\mathrel{\hspace{-0.1ex}\mathrel{\rlap{\raisebox{0.3ex}{$\m@th\cdot$}}\raisebox{-0.3ex}{$\m@th\cdot$}}\hspace{-0.73ex}\resizebox{0.55em}{0.84ex}{$=$}\hspace{0.67ex}}\hspace{-0.25em}}
\newcommand{\eqdef}{\hspace{-0.25em}\mathrel{\hspace{0.67ex}\resizebox{0.55em}{0.84ex}{$=$}\hspace{-0.73ex}\mathrel{\rlap{\raisebox{0.3ex}{$\m@th\cdot$}}\raisebox{-0.3ex}{$\m@th\cdot$}}\hspace{-0.1ex}}}
\makeatother
\usepackage{tikz}

%% file: Parini_Stylianou_2017_arXiv.bbl.references
\def\cprime{$'$}